\documentclass{article}

\usepackage[left=1in,right=1in,top=1in,bottom=1in]{geometry}

\usepackage{graphicx}
\usepackage{amssymb}
\usepackage{amsmath}
\usepackage{verbatim,booktabs}
\usepackage{subfigure,epsfig,url,psfrag}
\usepackage{float}
\usepackage{mathrsfs}

 \usepackage[usenames,dvipsnames]{pstricks}
 \usepackage{epsfig}
 \usepackage{pst-grad} 
 \usepackage{pst-plot} 
\usepackage{soul}
\usepackage{enumitem}

\makeatletter
\newcommand{\setwindow}[5]{
\def\xmin{#1}%
\def\ymin{#2}%
\def\xmax{#3}%
\def\ymax{#4}%
\pstFPsub\viewingwidth{#3}{#1}%
\pstFPdiv\result{\strip@pt#5}{\viewingwidth}%
\psset{unit=\result pt}}
\makeatother

\usepackage{amsthm}

\newtheorem{proposition}{Proposition}

\newtheorem{example}{Example}

\newtheorem{remark}{Remark}

\def\st{{\rm s.t.}}
\def\NP{{\mathcal NP}}

\def\ie{{i.e.,} }

\def\S{{\mathcal S}}

\def\I{{\mathcal I}}

\def\K{{\mathcal K}}
\def\I{{\mathcal I}}
\def\T{{\mathcal T}}

\def\P{{\mathcal P}}

\def\01{\ensuremath{0\mathord{-}1}}

\def\Ri{{\mathcal R}}

\def\MP{\text{MP}}

\def\LP{\text{LP}}

\newcommand{\bi}{\begin{list}{$\bullet$}{\setlength{\parsep}{0pt}\setlength{\itemsep}{0pt}}}

\newcounter{claim} 
\newenvironment{claim}[1][]
{\refstepcounter{claim} \begin{trivlist} \item[] {\em Claim~\theclaim.}\space#1 \itshape}
{\end{trivlist}}

\newcounter{mynotes}
\usepackage[ruled,vlined]{algorithm2e}

\begin{document}

\title{On the strength of recursive McCormick relaxations for binary polynomial optimization}

\author{
Aida Khajavirad
\thanks{Department of Industrial and Systems Engineering,
             Lehigh University.
             E-mail: {\tt aida@lehigh.edu}.
             }
}

\date{\today}

\maketitle

\begin{abstract}
Recursive McCormick relaxations are among the most popular convexification techniques for binary polynomial optimization. It is well-understood that both the quality and the size of these relaxations depend on the recursive sequence and finding an optimal sequence amounts to solving a difficult combinatorial optimization problem. We prove that any recursive McCormick relaxation is implied by the extended flower relaxation, a linear programming relaxation, which for binary polynomial optimization problems with fixed degree can be solved in strongly polynomial time.
\end{abstract}


\section{Binary polynomial optimization}
\label{sec:intro}
We consider a binary polynomial optimization problem of the form:
\begin{align}
\label{ML}
\tag{BP}
\begin{split}
\max & \quad \sum_{I\in \bar\I_0} {c^0_I \prod_{i\in I} {x_i}} \\
\st & \quad \sum_{I \in \bar\I_j} {c^j_I \prod_{i\in I} {x_i}} \leq b_j, \quad \forall j \in \{1,\ldots,m\}\\
& \quad x \in \{0,1\}^n,
\end{split}
\end{align}
where for each $j \in \{0,1,\ldots,m\}$, we denote by $\bar\I_j$ a family of nonempty subsets of $\{1,\ldots, n\}$, and $c^j_I$, $I\in \bar\I_j$ are nonzero real-valued coefficients.
We refer to each product term $\prod_{i\in I} {x_i}$ with $|I| \geq 2$ as a \emph{multilinear term}, and we refer to $r = \max\{|I| : I \in \cup_{j=0}^m \bar\I_j\}$ as the degree of Problem~\eqref{ML}. Throughout this paper, we assume that $r \geq 3$.
Problem~\eqref{ML} is $\NP$-hard even when $m=0$, and it has been the subject of extensive research by the mixed-integer nonlinear optimization (MINLP) community over the past three decades~\cite{yc93,s97,rs01,BorHam02, trx:10,lnl:12,bkst:14,dPKha16,dPIda16,CraRod17,dPIda18,dPIda19,dPIdaSah20,chenDashGun20}.

In this paper, we are interested in the quality of linear programming (LP) relaxations for Problem~\eqref{ML}. To convexify Problem~\eqref{ML}, it is common practice to first linearize the objective function and all constraints by introducing a new variable $y_I$ for every multilinear term $\prod_{i\in I} {x_i}$
and obtain an equivalent optimization problem in a lifted space:
\begin{align}
\label{EML}
\tag{EBP}
\begin{split}
\max & \quad \sum_{I\in \bar\I_0\setminus \I_0} {c^0_I x_I} + \sum_{I\in \I_0} {c^0_I y_I} \\
\st & \quad \sum_{I\in \bar\I_j\setminus \I_j} {c^j_I x_I} + \sum_{I\in \I_j} {c^j_I y_I} \leq b_j, \quad \forall j \in \{1,\ldots,m\}\\
& \quad y_I = \prod_{i\in I} {x_i}, \quad \forall I \in \bigcup_{j=0}^m{\I_j}\\
& \quad x \in \{0,1\}^n, \quad \forall i \in \{1,\ldots, n\},
\end{split}
\end{align}
where $\I_j$, $j \in \{0,1,\ldots,m\}$ consists of all elements in $\bar \I_j$ of cardinality at least two. Moreover, with a slight abuse of notation we set $x_{\{i\}}: =x_{i}$
Subsequently, a convex relaxation of the feasible region of Problem~\eqref{EML} is constructed and the resulting problem is solved to obtain an upper bound
on the optimal value of Problem~\eqref{ML}. The Reformulation Linearization technique (RLT)~\cite{st92} and the sum-of-squares hierarchy~\cite{parrilo03} are two general schemes for constructing LP and semi-definite programming (SDP) relaxations for polynomial optimization problems, respectively. While these methods often lead to strong bounds, they are too expensive to solve.
In the following, we focus on a number of convexification techniques that generate cheap LP relaxations for Problem~\eqref{EML} by building polyhedral relaxations for the set
\begin{equation}\label{set}
\S = \Big\{(x,y): \; y_I = \prod_{i\in I} {x_i}, \; \forall I \in \tilde \I := \bigcup_{j=0}^m{\I_j}, \; x \in \{0,1\}^n\Big\}.
\end{equation}
Throughout this paper, we refer to set $\S$ defined by~\eqref{set} as the~\emph{multilinear set} and we refer to its convex hull as the~\emph{multilinear polytope}.
Without loss of generality, we assume that each $x_i$, $i \in \{1,\ldots, n\}$ appears in at least one multilinear term.
Moreover, we refer to $r = \max\{|I| : I \in \tilde \I\}$ as the \emph{rank} of the multilinear set. Notice that the rank of a multilinear set equals to the degree of the corresponding binary polynomial optimization problem.

\subsection{Recursive McCormick relaxations}
Motivated by factorable programming techniques~\cite{mcc76}, a polyhedral relaxation of the multilinear set defined by~\eqref{set}
can be constructed by a recursive application of bilinear envelopes~\cite{kf83}. In the following we give a brief overview of this method in its full generality:
\begin{itemize}[leftmargin=*]
\item \emph{Initialization.} Initialize the sets $\Ri_I = \{I\}$ for all $I \in \tilde \I$, and $\tilde \K = \emptyset$. Mark each $L \in \Ri_I$ for all $I \in \tilde \I$ as \emph{unchecked}.
\item  \emph{Recursive Decomposition.} For each $I \in \tilde \I$, and for each $L \in \Ri_I$ such that $L \in \{I\} \cup \tilde \K$, consider the following cases:
       \begin{itemize}
           \item If $L$ is unchecked, then denote by $P_L =\{J, K\}$ a partition of $L$; that is, $J,K$ are nonempty, $J \cup K = L$ and $J \cap K = \emptyset$.
                 If $|J| >1 $ (resp. $|K| > 1$), let $\Ri_I = \Ri_I \cup \{J\}$ (resp. $\Ri_I = \Ri_I \cup \{K\}$).
                 Moreover, if $|J| > 1$ and $J \notin \tilde \I \cup \tilde \K$ (resp. $|K| > 1$ and $K \notin \tilde \I \cup \tilde \K$), then introduce the new \emph{artificial variable} $y_{J} = \prod_{i\in J} {x_i}$, update $\tilde \K = \tilde \K \cup \{J\}$, and mark $J$ as unchecked
                   (resp. $y_{K} = \prod_{i\in K} {x_i}$, update $\tilde \K = \tilde \K \cup \{K\}$, and mark $K$ as unchecked). Finally, mark $L$ as \emph{checked}.
            \item If $L$ is checked, then consider its partition $P_L = \{J,K\}$, constructed above.
                  If $|J| >1 $ (resp. $|K| > 1$), let $\Ri_I = \Ri_I \cup \{J\}$ (resp. $\Ri_I = \Ri_I \cup \{K\}$).
       \end{itemize}
\end{itemize}

It then follows that for each $L \in \Ri_I$ with $P_L = \{J,K\}$, we have $y_L = y_{J} y_{K}$, where we define $y_{\{i\}} := x_i$ for all $i \in \{1,\ldots,n\}$.
Therefore, the multilinear set $\S$ can be written, in a lifted space of variables, as a collection of bilinear equations of the form $y_{J_t \cup K_t} = y_{J_t}  y_{K_t}$, $t \in \{1,\ldots, T\}$. Subsequently a polyhedral relaxation of this set is obtained by replacing each bilinear equation by its convex hull over the unit hypercube:
\begin{align}\label{bienv}
\begin{split}
& y_{J_t \cup K_t}  \geq 0, \quad y_{J_t \cup K_t} \geq y_{J_t}+y_{K_t}-1, \\
& y_{J_t \cup K_t}  \leq y_{J_t}, \quad y_{J_t \cup K_t}  \leq y_{K_t},
\end{split} \quad \forall t \in \{1,\ldots,T\}.
\end{align}
This simple relaxation technique is used extensively in the literature and is often referred to as the \emph{recursive McCormick relaxation} (RMC).
In the remainder of this paper, given an RMC of the multilinear set $\S$, we refer to $\Ri_I$ as the~\emph{recursive sequence} of the multilinear term $y_I = \prod_{i \in I}{x_i}$, $I \in \tilde \I$
and we refer to $\Ri_I$, $I \in \tilde \I$ as the recursive sequence of $\S$.
Moreover, by the~\emph{size} of an RMC, we imply the number of its artificial variables; \ie $|\tilde \K|$.
For notational simplicity, throughout the paper, for any subset of $\{1,\cdots,n\}$, say for example $\{i,j,k\}$, we write $y_{ijk}$ instead of $y_{\{i,j,k\}}$; similarly,
we write $\Ri_{ijk}$ instead of $\Ri_{\{i,j,k\}}$.

For instance, the global solver {\tt BARON}~\cite{IdaNick18} uses an RMC scheme in which for each $y_{J_t \cup K_t} = y_{J_t}  y_{K_t}$, it chooses $J_t,K_t$
such that $|J_t|=1$ (see~\cite{tawsah04} for further details).  For example, consider the multilinear set $\S=\{(x,y): y_{1234}= x_1 x_2 x_3 x_4, \; x \in \{0,1\}^4\}$; then the factorable decomposition module in {\tt BARON} constructs the following bilinear equations $y_{1234}= x_1 y_{234}, y_{234} = x_2 y_{34}, y_{34} = x_3 x_4$; in this case, the recursive sequence of $y_{1234}$ is given by $\Ri_{1234} = \{\{1,2,3,4\},\{2,3,4\},\{3,4\}\}$ and artificial variables of this RMC are $y_{234}, y_{34}$.

For a multilinear set, there exist many different recursive sequences, in general. It is well-understood that the choice of the recursive sequence impacts both the size and the quality of the resulting RMC.
In~\cite{ragh22}, the authors study the impact of the recursive sequence on the size of and the quality of the resulting RMC. First, the authors show that even for multilinear sets of rank three, the problem of finding a minimum-size RMC is $\NP$-hard. Moreover, they propose a mixed-integer programming (MIP) formulation to solve this problem. Subsequently, they propose a MIP formulation to identify RMCs of a given size with the best possible upper bound. Their numerical experiments suggest that their proposed RMCs often outperform the simpler RMCs implemented in global solvers. To better illustrate this point, we use the running example of~\cite{ragh22} which we will later use to explain our results as well:

\begin{example}\label{ex1}
Consider the problem of maximizing $f(x)=-x_1x_2x_3 + x_2x_3x_4+x_1x_3 x_4$ over $x \in \{0,1\}^4$. It can be checked that if we solve an RMC with the recursive sequence
$$\Ri^1_{123} = \{\{1,2,3\}, \{1,2\}\}, \; \Ri^1_{234} = \{\{2,3,4\},\{2,3\}\}, \; \Ri^1_{134} = \{\{1,3,4\}, \{1,3\}\},$$
we obtain the upper bound $U= \frac{4}{3}$, while if we solve an RMC with the recursive sequence
$$\Ri^2_{123} = \{\{1,2,3\},\{1,3\}\}, \; \Ri^2_{234} = \{\{2,3,4\},\{3,4\}\}, \; \Ri^2_{134} = \{\{1,3,4\}, \{3,4\}\},$$
we obtain the upper bound $U=1.0$, which is sharp.  The recursive sequence $\Ri^2$ outperforms $\Ri^1$ since the artificial variable $y_{34}$ introduced in $\Ri^2$ appears in the recursive sequence of two  multilinear terms $y_{234} = x_2x_3x_4$, $y_{134} = x_1x_3 x_4$, while all artificial variables of $\Ri^1$ appear in the recursive sequence of one multilinear term.
\end{example}

\subsection{Multilinear polytopes and hypergraph acyclicity}
In~\cite{dPKha16}, we introduce a hypergraph representation of multilinear sets which in turn
enables us to characterize multilinear polytopes of hypergraphs with various degrees of acyclicity~\cite{fag83} in the space of original variables~\cite{dPIda18,dPIda19}. As a result, we obtain strong polyhedral relaxations for general multilinear sets~\cite{dPIdaSah20}.
In the following, we briefly review two polyhedral relaxations of multilinear sets and introduce a new stronger polyhedral relaxation; in the next section, we compare these relaxations with RMCs.

Before proceeding further, we introduce our hypergraph terminology.
Recall that a hypergraph $G$  is a pair $(V,E)$ where $V=V(G)$ is the set of nodes of $G$, and $E=E(G)$
is a set of subsets of $V$ of cardinality at least two, called the edges of $G$.
With any hypergraph $G$, we associate a multilinear set $\S_G$ defined as follows:
\begin{equation} \label{eq: SG}
\S_G= \Big\{ z \in \{0,1\}^d : z_e = \prod_{v \in e} {z_{v}}, \; e \in E \Big\},
\end{equation}
where $d = |V|+|E|$.
We denote by $\MP_G$ the convex hull of $\S_G$.
Note that variables $z_{v}$, $v \in V$ in~\eqref{eq: SG} correspond to variables $x_i$, $i \in \{1,\ldots,n\}$ in~\eqref{set}
and variables $z_e$,  $e \in E$ in~\eqref{eq: SG} correspond to variables $y_I$, $I \in \tilde \I$ in~\eqref{set}.

\paragraph{The standard linearization.} A simple polyhedral relaxation of $\S_G$ can be obtained by replacing each multilinear term
$z_e  = \prod_{v\in e}{z_v}$, by its convex hull over the unit hypercube:
\begin{eqnarray}\label{stdre}
\MP^{\LP}_G & =\Big\{z:  & z_v \leq 1, \; \forall v \in V, \nonumber \\
&  & z_e  \geq 0, \; z_e \geq \sum_{v\in e}{z_v}-|e|+1, \; \forall e \in E,\\
&  & z_e  \leq z_v, \forall v \in e, \; \forall e \in E \Big\}.\nonumber
\end{eqnarray}
The above relaxation has been used extensively in the literature and is often referred to as the~\emph{standard linearization}
of the multilinear set (see for example~\cite{gw:74,yc93}). In~\cite{dPIda18}, we prove that $\MP^{\LP}_G  = \MP_G$ if and only if $G$ is a Berge-acyclic hypergraph; \ie the most restrictive type of hypergraph acyclicity~\cite{fag83}.
For general hypergraphs however, the standard linearization often provides a weak relaxation of multilinear sets.
As we show in the next section, the standard linearization corresponds to the projection of an RMC, if the common product terms across various multilinear terms are not exploited. The following example demonstrates this result:

\begin{example}\label{ex2}
Consider the binary polynomial optimization problem considered in Example~\ref{ex1}. It can be checked that using the standard linearization of the corresponding
multilinear set, we obtain an upper bound $U=\frac{4}{3}$, which is equal to the bound obtained using the RMC with no artificial variable appearing in the recursive sequence of  more than one multilinear term.
\end{example}

\paragraph{The flower relaxation and the extended flower relaxation.} In~\cite{dPIda18}, we introduce the flower relaxation of multilinear sets, a polyhedral relaxation that is often significantly stronger than the
standard linearization. To formally define this relaxation, we first need to introduce flower inequalities~\cite{dPIda18}.
Let $e_0$ be an edge of $G$ and let $e_k$, $k \in K$, be the set of all edges adjacent to $e_0$ with
\begin{equation}\label{oldCond0}
|e_0 \cap e_k| \geq 2, \quad \forall k \in K.
\end{equation}
Let $T$ be a nonempty subset of $K$ such that
\begin{equation}\label{oldCond}
e_0 \cap e_i\cap e_j = \emptyset, \quad \forall \; i\neq j\in T.
\end{equation}
Then the~\emph{flower inequality} centered at $e_0$ with neighbors $e_k$, $k \in T$, is given by:
\begin{equation}\label{flowerIneq}
\sum_{v \in e_0\setminus \cup_{k\in T} {e_k}}{z_v}+\sum_{k \in T} {z_{e_k}} - z_{e_0} \leq |e_0\setminus \cup_{k\in T} {e_k}|+|T|-1.
\end{equation}
We refer to inequalities of the form~\eqref{flowerIneq}, for all nonempty $T \subseteq K$
satisfying conditions~\eqref{oldCond0} and~\eqref{oldCond}, as the system of flower inequalities centered at $e_0$.
We define the \emph{flower relaxation} $\MP^F_G$ as the polytope obtained by adding the system of flower inequalities centered at each edge of $G$ to $\MP^{\LP}_G$.
In~\cite{dPIda18}, we prove that $\MP^F_G =\MP_G$ if and only if $G$ is a $\gamma$-acyclic hypergraph. We should remark that $\gamma$-acyclic hypergraphs represent
a significant generalization of Berge-acyclic hypergraphs~\cite{fag83}. As we detail in~\cite{dPIdaSah20}, while the separation problem over the flower relaxation is $\NP$-hard for general hypergraphs, it can be solved in $O(|E|^2)$ time for hypergraphs with fixed rank.

\begin{example}\label{ex3}
Consider the binary polynomial optimization problem considered in Example~\ref{ex1}. It can be checked that using the flower relaxation of the corresponding
multilinear set, we obtain an upper bound $U= 1.0$, which coincides with the best RMC bound.
\end{example}

While in Example~\ref{ex3} the flower relaxation is as good as the best RMC, this is not always the case.
Next, we introduce a stronger polyhedral relaxation of multilinear sets that dominates all RMC relaxations.
To this end, we first present a class of valid inequalities for multilinear sets that can be considered as a generalization of flower inequalities.
Let $e_0$ be an edge of $G$ and as before, denote by $e_k$, $k \in K$, the set of all edges adjacent to $e_0$.
Let $T$ be a nonempty subset of $K$ such that
\begin{equation}\label{newCond}
\gamma_i :=\Big|(e_0 \cap e_i)\setminus \bigcup_{j \in T \setminus \{i\}}{(e_0 \cap e_j)}\Big| \geq 2, \quad \forall i \in T.
\end{equation}
Then the~\emph{extended flower inequality} centered at $e_0$ with neighbors $e_k$, $k \in T$, is given by~\eqref{flowerIneq}.
Similarly, we define the \emph{extended flower relaxation} $\MP^{EF}_G$ as the polytope obtained by adding the system of extended flower inequalities centered at each edge of $G$ to $\MP^{\LP}_G$. Notice that conditions~\eqref{oldCond0} and~\eqref{oldCond} imply condition~\eqref{newCond}; namely, flower inequalities are a special case of extended flower inequalities. Indeed, extended flower inequalities are a generalization of flower inequalities as they subsume cases in which $e_0 \cap e_i\cap e_j \neq \emptyset$
for some $i\neq j\in T$.

We should remark that in the special case where the neighbours $e_k$, $k \in T$ satisfy the so-called~\emph{running intersection property}, it can be checked that~\emph{running intersection inequalities}~\cite{dPIda19} imply extended flower inequalities. However, in general, extended flower inequalities are not implied by running intersection inequalities. The following example illustrates these facts:

\begin{example}
Consider the hypergraph $G = (V, E)$ with $V = \{v_1, \ldots, v_9\}$ and $E = \{e_0, e_1, e_2, e_3\}$, where $e_0 = V$, $e_1 = \{v_1, \ldots, v_4\}$,
$e_2 = \{v_4, \ldots, v_7\}$, and $e_3 = \{v_1, v_7, v_8, v_9\}$. Then all flower inequalities of $\S_G$ are given by:
\begin{align}\label{fls}
\begin{split}
&\sum_{v \in e_0 \setminus e_i}{z_v}+ z_{e_i} - z_{e_0} \leq 5, \\
&\quad z_{e_0} - z_{e_i} \leq 0,
\end{split} \quad \forall i \in \{1,2,3\},
\end{align}
and the remaining running intersection inequalities of $\S_G$ are given by
\begin{equation}\label{rIs}
-z_{e_i \cap e_j} + \sum_{v \in e_0 \setminus (e_i \cup e_j)}{z_v}+ z_{e_i}+z_{e_j}-z_{e_0} \leq 2, \quad \forall i \neq j \in \{1,2,3\}.
\end{equation}
Moreover, in addition to inequalities~\eqref{fls}, we have the following extended flower inequalities:
\begin{align}
& \sum_{v \in e_0 \setminus (e_i \cup e_j)}{z_v}+ z_{e_i}+z_{e_j}-z_{e_0} \leq 3, \quad \forall i \neq j \in \{1,2,3\}\label{exf11}\\
& \qquad z_{e_1}+ z_{e_2}+z_{e_3}-z_{e_0} \leq 2. \label{exf12}
\end{align}
First notice that using inequalities $z_v \leq 1$ for all $v \in V$, which are present in the standard linearization of $\S_G$, we deduce that
extended flower inequalities~\eqref{exf11} are implied by the running intersection inequalities~\eqref{rIs}. However, the extended flower inequality~\eqref{exf12}
is not implied by any other inequalities. To see this consider the point $\tilde z_p$, $p \in V \cup E$ defined as:
\begin{align}\label{pt}
& \tilde z_{v} = 1, \quad \forall v \in \bar V: = \{v_1, v_4, v_7\}\nonumber\\
& \tilde z_{v} =\tilde z_{e} = \frac{3}{4}, \quad \forall v \in V \setminus \bar V, \; \forall e \in \{e_1, e_2, e_3\}\\
& \tilde z_{e_0} = 0 \nonumber.
\end{align}
Clearly, $\tilde z$ violates the extended flower inequality~\eqref{exf12}, since $3 (\frac{3}{4}) \not\leq 2$.
However, it can be checked that $\tilde z$ satisfies all inequalities present in
the standard linearization of $\S_G$, as well as flower inequalities~\eqref{fls}, and running intersection inequalities~\eqref{rIs}. This is due to the fact that while the set $\{e_i, e_j\}$, has the running intersection property
for all $i \neq j \in \{1,2,3\}$, the set $\{e_1, e_2, e_3\}$ do not have the running intersection property and hence no running intersection inequality can be generated for the latter. See~\cite{dPIda19} for details regarding running intersection inequalities.
\end{example}

\begin{remark}\label{rem1}
It is important to note that assumption~\eqref{newCond} is without loss of generality. To see this, consider an edge $e_0$
and a subset of adjacent edges $e_k$, $k \in T$ some of which do not satisfy assumption~\eqref{newCond}. The following cases arise:
\begin{itemize}
\item $\gamma_i = 0$ for some $i \in T$. In this case, the extended flower inequality is implied by two inequalities: (i) $z_{e_i} \leq 1$,
and (ii) the extended flower inequality centered at $e_0$ with neighbours $e_k$, $k \in T \setminus \{i\}$.

\item $\gamma_i = 1$ for some $i \in T$. In this case, the extended flower inequality is implied by two inequalities: (i) $z_{e_i} \leq z_{\bar v}$,
where $\{\bar v\} = e_0 \cap e_i$, and (ii) the extended flower inequality centered at $e_0$ with neighbours $e_k$, $k \in T \setminus \{i\}$.
\end{itemize}
Hence we conclude that to construct the extended flower relaxation, it suffices to consider neighbours $e_k$, $k \in T$ satisfying assumption~\eqref{newCond}.
\end{remark}

We conclude this section by commenting on the complexity of optimizing over the extended flower relaxation.
Clearly, optimization over the standard linearization can be done in polynomial time. Hence it suffices to consider extended flower inequalities.
In~\cite{dPIdaSah20}, the authors prove that the separation problem over flower inequalities is $\NP$-hard.
However, they show that for hypergraphs of fixed degree, the separation problem can be solved in strongly polynomial time.

\paragraph{Separating over extended flower inequalities.} In the following, we present a separation algorithm over extended flower inequalities that runs
in strongly polynomial-time for hypergraphs $G=(V,E)$ of fixed degree $r$; \ie in a number of
iterations bounded by a polynomial in $|V |$ and $|E|$. By equivalence of separation and optimization, this in turn implies that optimizing over
the extended flower relaxation can be done in polynomial time. The fixed degree assumption is reasonable as for binary polynomial optimization problems, $|V|$
and $|E|$ represent the number of variables and multilinear terms, respectively, while $r$ is the degree of the polynomial.
Indeed, for all problems appearing in applications, while we have hundreds or thousands of variables, the degree of the polynomial is a small number; namely,
$r \ll 10$.  In~\cite{dPIda19}, we present an efficient separation algorithm for running-intersection inequalities. While the extended flower inequalities do not imply and are not implied by running-intersection inequalities, they can be separated using similar ideas.
Hence, in the following, we provide a sketch of our separation algorithm and
refer the reader to~\cite{dPIda19} for further details.

Let us start by formally defining the separation problem: Given a hypergraph $G = (V,E)$, and a vector $\bar z \in [0, 1]^{V+E}$,
decide whether $\bar z$ satisfies all extended flower inequalities or not,
and in the latter case, find an extended flower inequality that is violated by $\bar z$.

Consider an edge $e_0 \in E$, and denote by $E_{e_0}$ the subset of $E$ containing, for every $f \subseteq e_0$ with $|f| \ge 2$,
among all the edges $e \in E \setminus \{e_0\}$ with $e_0 \cap e = f$, only one that maximizes $\bar z_{e}$.
It then follows that in order to solve the separation problem over all extended flower inequalities in $G$ centered at $e_0$, it suffices
to do the separation over the (much smaller) set of extended flower inequalities centered $e_0$, with neighbors $e_k$ contained in $E_{e_0}$.
In the following, we assume that a rank-$r$ hypergraph $G = (V,E)$ is represented
by an incidence-list in which every edge
stores the vertices it contains. We assume the vertex list for each edge is sorted, otherwise, such a representation
can be constructed in $O(r|E|)$ operations using some integer sorting algorithm. We are now ready to present our separation algorithm:

\begin{proposition} \label{sepAlg}
Consider a rank-$r$ hypergraph $G = (V, E)$ and a vector $\bar z \in [0,1]^{V+E}$. Then the separation problem over all extended flower inequalities
can be solved in $O(|E|(r2^r|E|+r^2 2^{r^2/2}))$ operations.
\end{proposition}

\begin{proof}
Let $e_0 \in E$; we present the separation algorithm over all extended flower inequalities centered at $e_0$.
By applying the algorithm $|E|$ times, we can solve the separation problem over all extended flower inequalities in $G$.

Observe that by assumption~\eqref{newCond}, the number of neighbours $t$ in an extended flower inequality does not exceed $\frac{r}{2}$.
First we construct the set $E_{e_0}$ defined above. For a rank-$r$ hypergraph $G$, the number of elements in $E_{e_0}$ is at most $2^r-r-1$.
It can be checked that the set $E_{e_0}$ can be constructed in $O(r 2^r |E|)$ operations.
It then follows that the total number of possible sets of neighbors of cardinality at most $\frac{r}{2}$ is given by $N =\sum_{i=1}^{r/2} \binom{2^r-r-1}{i} \leq (2^r-r)^{r/2}$.
For each possible set of $t$ neighbors, we can check the validity of condition~\eqref{newCond} in $O(t r)$ operations.
Finally, for each valid set of $t$ neighbors, we generate the corresponding unique extended flower inequality in $O(t + r)$ operations.
Therefore, the total running time of the separation algorithm is given by $O(|E|(r2^r|E|+r^2 2^{r^2/2}))$.
\end{proof}

Hence by Proposition~\ref{sepAlg}, for a fixed-rank hypergraph $G=(V,E)$, the separation problem over all extended flower inequalities
can be solved in $O(|E|^2)$ operations.  As we show in the next section, the extended flower relaxation implies~\emph{all} RMCs of a multilinear set. This is an important result as it indicates instead of solving a MIP to identify an optimal RMC and subsequently solving an LP in a lifted space (due to the addition of artificial variables),
one can solve the extended flower relaxation efficiently and obtain no worse (in fact, often better) bounds.

\section{A theoretical comparison of relaxations}
\label{sec:McFl}

In this section, we present a theoretical comparison of RMCs, the standard linearization, and the extended flower relaxation of multilinear sets.
Let $G=(V,E)$ be a hypergraph, and consider an RMC for the multilinear set $\S_G$. We denote by $\bar E$ the set containing the index set of 
all artificial variables corresponding to the RMC. That is, $z_e$, $e \in \bar E$ denote all artificial variables of this RMC. 
From the definition of RMCs it follows that $E \cap \bar E = \emptyset$, and $|e| \geq 2$ for all $e \in \bar E$.
With any RMC of $S_G$, we then associate a hypergraph $\bar G = (V, E \cup \bar E)$, and we refer to $\bar E$ as the set of artificial edges.
We denote by $\MP^{\rm RMC}_G$ the projection of the RMC onto the original space, \ie
the space of variables $z_p$, $p \in V \cup E$.

\subsection{Recursive McCormick relaxations versus the standard linearization}

Given an RMC of a multilinear set $\S_G$, in the following, we characterize the conditions under which optimizing over this RMC
is equivalent to optimizing over the standard linearization $\MP^{\rm LP}_G$ defined by~\eqref{stdre}.
We say that an RMC for $\S_G$ with the recursive sequence $\Ri_e$, $e \in E$ is \emph{non-overlapping}, if $\Ri_e \cap \Ri_{e'} = \emptyset$ for all $e \neq e' \in E$.  For example, the RMC of Example~\ref{ex1} with the recursive sequence $\Ri^1$ is non-overlapping since $\Ri_{123}^1 \cap \Ri_{234}^1 = \Ri_{123}^1 \cap \Ri_{134}^1 = \Ri_{134}^1 \cap \Ri_{234}^1 = \emptyset$, whereas, the recursive sequence $\Ri^2$ is not non-overlapping, since $\Ri^2_{234} \cap \Ri^2_{134} = \{\{3,4\}\}$. The following proposition states the relationship between RMCs and the standard linearization of multilinear sets.

\begin{proposition}\label{prop1}
Let $G=(V,E)$ be a hypergraph and consider an RMC for the multilinear set $\S_G$.
Then $\MP^{\rm RMC}_G =\MP^{\rm LP}_G$, if and only if the RMC is non-overlapping.
\end{proposition}

\begin{proof}
First suppose that the RMC is non-overlapping with the recursive sequence $\Ri_e$, $e \in E$.  Consider an edge $\tilde e \in E$.
By the non-overlapping assumption, we have the following two properties:
\begin{itemize}
\item  There exists no $e \in \Ri_{\tilde e}$ such that $e \neq \tilde e$ and $e \in E$, as otherwise we have $\Ri_{\tilde e} \cap \Ri_{e} = \{e\}$, which contradicts the non-overlapping assumption.

\item  For any $e' \in \Ri_{\tilde e}$, we have $e' \notin \Ri_{e}$ for all $e \in E \setminus \{\tilde e\}$.
\end{itemize}
By above observations, for each $\tilde e \in E$, all variables $z_{\bar e}$, $\bar e \in \Ri_{\tilde e} \setminus \{\tilde e\}$ are artificial variables and these variables do not appear in any other $\Ri_{e}$ with $e \neq \tilde e$. Hence, for each $e \in E$, our task is to consider inequalities containing only variables $z_e$, $e \in \Ri_e$
and to project out all $z_{\bar e}$ with $\bar e \subset e$ from these inequalities.
Since the standard linearization is obtained by replacing each $z_e = \prod_{v \in e} {z_v}$, $z_v \in \{0,1\}$, $e \in E$, by its convex hull,
we deduce that $\MP^{\rm LP}_G \subseteq \MP^{\rm RMC}_G$.
Hence it suffices to show that $\MP^{\rm RMC}_G \subseteq \MP^{\rm LP}_G$. Consider a point $\tilde z_p$, $p \in V \cup E$ in $\MP^{\rm RMC}_G$.
Recall that $z_e$, $e \in \bar E$ denotes the artificial variables of the RMC.
By definition of $\MP^{\rm RMC}_G$ there exists $\tilde z_e$, $e \in \bar E$ such that $\tilde z_p$, $p \in V \cup E \cup \bar E$ satisfies the RMC.
First, from all inequalities of the form $z_{p_1 \cup p_2} \leq z_{p_1}$, $z_{p_1 \cup p_2} \leq z_{p_2}$ in this RMC, we deduce that $\tilde z_{e} \leq \tilde z_{v}$
for all $v \in e$ and for all $e \in E$. Second, from all inequalities of the form  $z_{p_1 \cup p_2} \geq z_{p_1} + z_{p_2}-1$ and $z_{p_1 \cup p_2} \leq z_{p_1}$, we deduce that
$\tilde z_v \leq 1$ for all $v \in V$. Finally, from all inequalities of the form  $z_{p_1 \cup p_2} \geq z_{p_1} + z_{p_2}-1$ in this RMC we deduce that
$\tilde z_e \geq \sum_{v \in e}{\tilde z_v}-|e|+1$ for all $e \in E$. Notice that the RMC includes inequalities $z_e \geq 0$ for all $e \in E$. We conclude that
$\tilde z_p$, $p \in V \cup E$ is present in $\MP^{\rm LP}_G$. Therefore, $\MP^{\rm RMC}_G \subseteq \MP^{\rm LP}_G$, implying if the RMC is non-overlapping, we have $\MP^{\rm RMC}_G = \MP^{\rm LP}_G$.

\smallskip

Next, suppose that the RMC is not non-overlapping; \ie $\Ri_{e'} \cap \Ri_{e''} = \{\tilde e\}$ for some $e',e'' \in E$ and some $\tilde e \in E \cup \bar E$.
Two cases arise:
\begin{itemize}
\item [(i)] If $\tilde e \in E$, then without loss of generality, let $\tilde e=e''$. It then follows that the RMC implies the inequality $z_{e'} \leq z_{e''}$,
while this inequality is not implied by the standard linearization. To see this, consider the point $\tilde z_v = \tilde z_e = \frac{1}{2}$ for all $v \in V$, $e \in E \setminus \{e''\}$, and $\tilde z_{e''} = 0$. It can be checked that $\tilde z \in \MP^{\rm LP}_G$, while $\tilde z_{e'} > \tilde z_{e''}$.

\item [(ii)] If $\tilde e \in \bar E$, then the RMC implies the following inequalities $z_{e'} \leq z_{\tilde e}$ and
$z_{e''} \geq z_{\tilde e} + \sum_{v \in e'' \setminus \tilde e}{z_v}-|e'' \setminus \tilde e|$. Projecting out $z_{\tilde e}$ from these two inequalities, we conclude that
the inequality
\begin{equation}\label{flaux}
\sum_{v \in e'' \setminus \tilde e}{z_v}+ z_{e'} - z_{e''} \leq |e'' \setminus \tilde e|,
\end{equation}
is implied by $\MP^{\rm RMC}_G$. We argue that inequality~\eqref{flaux} is not implied by $\MP^{\rm LP}_G$.
To see this, consider the point $\tilde z$ defined as:  $\tilde z_v = 1$ for all $v \in e'' \setminus \tilde e$,
$\tilde z_v = \frac{1}{2}$ for all $v \in (e' \setminus e'') \cup \tilde e$, $\tilde z_v = 0$ for the remaining nodes in $G$,
$\tilde z_{e'} = \frac{1}{2}$, $\tilde z_{e''} = 0$, $\tilde z_e = 1$ for all $e \subseteq e'' \setminus \tilde e$, $\tilde z_e = 0$ for all $e \nsubseteq e' \cup e''$
and $\tilde z_e = \frac{1}{2}$ for all remaining edges in $G$.
This point does not satisfy inequality~\eqref{flaux}, as $|e'' \setminus \tilde e| + \frac{1}{2} - 0\nleq |e'' \setminus \tilde e|$. However, it can be checked that $\tilde z$ belongs to $\MP^{\LP}_G$, provided that $|\tilde e \cap e''| \geq 2$, which always holds since $\tilde e \subset e''$ and $|\tilde e| \geq 2$.
\end{itemize}
By~(i) and~(ii) above, it follows that if $\MP^{\rm RMC}_G =\MP^{\rm LP}_G$, then the RMC is non-overlapping, and this completes the proof.
\end{proof}

Roughly speaking, Proposition~\ref{prop1} indicates that the standard linearization is equivalent to the \emph{weakest} type of RMCs; \ie RMCs not exploiting common product terms in various multilinear terms. Namely, when such common expressions exist, one can always construct RMCs that are stronger than the standard linearization (see Examples~\ref{ex1} and~\ref{ex2}).

\subsection{Recursive McCormick relaxation versus the extended flower relaxation}
We now present the main result of this paper. Namely, the extended flower relaxation implies all RMCs of a multilinear set.
In fact, as we show via an example, the extended flower relaxation often provides a stronger relaxation than the best RMC.
For notational simplicity, in the following, for any node $v \in V$, instead of $z_{\{v\}}$, we write $z_v$.

\begin{proposition}\label{prop2}
Let $G=(V,E)$ be a hypergraph, and consider the multilinear set $\S_G$.
Then $\MP^{\rm EF}_G \subseteq \MP^{\rm RMC}_G$ for all RMCs of $\S_G$.
\end{proposition}

\begin{proof}
Consider an RMC of $\S_G$; we assume that this RMC does not satisfy the non-overlapping property as otherwise the result follows from Proposition~\ref{prop1}.
For each $\tilde e \in E$ and each artificial edge $\bar e \in  \Ri_{\tilde e} \cap\bar E$, define
\begin{equation}\label{e8}
\P_{\tilde e, \bar e} := \{e\in E \setminus \{\tilde e\}: \; \bar e \in \Ri_e\}.
\end{equation}
For each $e \in E$, denote by $T_e$ a partition of $e$ such that each $p \in T_e$ satisfies $p \in e \cup \Ri_e$
and for $p \in T_e \cap \bar E$, we have $\P_{e,p} \neq \emptyset$. Denote by $w_{e,p}$ be an element of $\P_{e,p}$.
Recall that by definition of partition we have $p \cap p' = \emptyset$ for all $p,p' \in T_e$ and $\cup_{p \in T_e} p = e$.
We refer to any such partition $T_e$ of $e$ as a \emph{flower partition} of $e$.
Let $\bar T_e$ be a subset of $T_e$ consisting of all artificial edges.
Denote by $\T_e$ the set consisting of all flower partitions of $e$.
To complete the proof it suffices to prove the following claim:

\begin{claim}\label{cl1}
The polytope $\MP^{\rm RMC}_G$ is defined by the following inequalities:
\begin{align}
& z_v \leq 1, \quad \forall v \in V, \quad z_e \geq 0, \quad \forall e \in E \label{proj0}\\
& z_e \leq z_p, \quad \forall p \in e \cup (\Ri_e \cap E), \;\forall e \in E \label{proj1}\\
& \sum_{p \in T_e \setminus \bar T_e}{z_p}+ \sum_{p \in \bar T_e}{z_{w_{e,p}}} - z_e \leq |T_e|-1, \quad \forall w_{e,p} \in \P_{e,p}, \; \forall T_e \in \T_e, \; \forall e \in E.\label{proj2}
\end{align}
\end{claim}

Notice that all above inequalities are implied by $\MP^{\rm EF}_G$. That is, inequalities~\eqref{proj0} are present in the standard linearization;
inequalities~\eqref{proj1} are flower inequalities centered at $p$, if $p \in E$, and
are present in the standard linearization, if $p \in e$. Inequalities~\eqref{proj2} are extended flower inequalities (or are implied by extended flower inequalities, see Remark~\ref{rem1}) centered at $e$, if there exists some $p \in T_e$
with $|p| \geq 2$, and are present in the standard linearization otherwise, as in the latter case they simplify to $\sum_{v \in e}{z_v}-z_e \leq |e|-1$.

To prove Claim~\ref{cl1} we need to introduce some notation.
Consider an artificial edge $\bar e \in \bar E$ and consider any $e \in E$ such that $\bar e \in \Ri_e$.
We define the~\emph{level} of the artificial variable $z_{\bar e}$ as
$$t_{\bar e}:= \Big|\{e' \in \Ri_e \cap \bar E:     \bar e \supseteq e'\}\Big|.$$
Note that from the definition of $t_e$ it follows that $1 \leq t_e \leq r-2$ for all $e \in \bar E$, where as before $r$ denotes the rank of $\S_G$.
Moreover, we define \emph{the level of an RMC} as
$$t := \max_{e \in \bar E}{t_e},$$
and we define $t := 0$, if the RMC contains no artificial variables.

\medskip

\emph{Proof of Claim~1.} We prove by induction on level $t$ of the RMC:

\emph{The base case.} Consider an RMC of level zero; that is, the RMC does not contain any artificial variables. Then all inequalities defining the RMC
are of the following forms:
\begin{align*}
& z_{p \cup q} \geq 0,\\
& z_{p \cup q} - z_{p} \leq 0,\\
& z_{p \cup q} - z_{q} \leq 0,\\
& z_{p} + z_{q} - z_{p \cup q} \leq 1,
\end{align*}
where $p, q \in V \cup E$ and $p \cup q \in E$. All above inequalities are present in System~\eqref{proj0}-\eqref{proj2}.
The first inequality is of the form~\eqref{proj0}, the second and third inequalities are of the form~\eqref{proj1} and, the fourth inequality is of the form~\eqref{proj2}
with $\bar T_e = \emptyset$ and $|T_{e}| = 2$ for all $e \in E$.

\medskip

\emph{The inductive step.} Now consider an RMC of level $k$ for the multilinear set $\S_G$, $G=(V, E)$ with artificial variables $z_e$, $e \in \bar E$.
Let $z_e$, $e \in \tilde E$ denote all level-1 artificial variables
in this RMC.
If $\tilde E = \emptyset$, then the proof follows from the base case. Henceforth, suppose that $\tilde E \neq \emptyset$.
It can be checked that this RMC is equivalent to an RMC of level $k-1$ for the multilinear set $\S_{G'}$, $G' = (V, E')$,
where $E' = E \cup \tilde E$ and with artificial variables $z_e$, $e \in \bar E \setminus \tilde E$.
Denote by $T'_e$ a flower partition of $e \in E \cup \tilde E$ in this RMC and denote by $\bar T'_e = T'_e \cap (\bar E \setminus \tilde E)$.
Therefore, by the induction hypothesis, the polytope $\MP^{\rm RMC}_{G'}$ is defined by the following inequalities:
\begin{align}\label{projk}
& z_v \leq 1, \quad \forall v \in V\nonumber\\
& z_e \geq 0, \quad \forall e \in E \cup \tilde E \nonumber\\
& z_e \leq z_p, \quad \forall p \in e \cup (\Ri_e \cap (E \cup \tilde E)), \;\forall e \in E \cup \tilde E\\
& \sum_{p \in T'_e \setminus \bar T'_e}{z_p}+ \sum_{p \in \bar T'_e}{z_{w_{e,p}}} - z_e \leq |T'_e|-1, \quad \forall w_{e,p} \in \P'_{e,p}, \; T'_e \in \T'_e, \; \forall e \in E \cup \tilde E,\nonumber
\end{align}
where as before $\T'_e$ denotes the set of all flower partitions of $e$, and $\P'_{e,p} = \{e' \in E \cup \tilde E \setminus \{e\}: p \in \Ri_{e'}\}$.
It can be checked that any inequality in System~\eqref{projk} not containing any artificial variable $z_e$, $e \in \tilde E$ is also present in System~\eqref{proj0}-\eqref{proj2}.
Hence to construct $\MP^{\rm RMC}_{G}$, it suffices to project out the remaining artificial variables $z_e$, $e \in \tilde E$ from System~\eqref{projk}.
Notice that all artificial edges $e \in \tilde E$ are level-1 edges; \ie for each $e \in \tilde E$, we have $e = p_e \cup q_e$ for some $p_e, q_e \in V \cup E$.
As a result, there exists no $e \neq e' \in \tilde E$ such that $e' \in \Ri_e$, \ie $T'_e \cap \tilde E= \emptyset$ for all $T'_e \in \T'_e$ and for all $e \in \tilde E$.
Now consider the RMC of $\S_{G'}$ defined above and for each $e \in E$, consider a flower partition $T'_e$ of $e$ for which there exists some $\tilde e \in \tilde E$ satisfying $\tilde e \in T'_e$.
Define $\tilde T_e := T'_e \cap \tilde E$. It can be checked that $T'_e$ coincides with a flower partition $T_e$ of $e$ in the RMC for $\S_G$ with the only difference that
the subset of $T_e$ containing edges of the artificial variables are $\bar T_e = \bar T'_e \cup \tilde T_e$.
It then follows that any artificial variable $z_{\tilde e}$, $\tilde e \in \tilde E$ only
appears in the following inequalities of System~\eqref{projk}:
\begin{align}
& z_{\tilde e} \geq 0 \label{k1}\\
& z_{\tilde e} \leq z_{p_{\tilde e}}, \; z_{\tilde e} \leq z_{q_{\tilde e}}\label{k2}\\
& z_{p_{\tilde e}} + z_{q_{\tilde e}} - z_{\tilde e} \leq 1\label{k3}\\
& z_e \leq z_{\tilde e}, \quad \forall e \in E: \tilde e \in \Ri_e\label{k4}\\
& \sum_{p \in \tilde T_e}{z_p}+\sum_{p \in T_e \setminus \bar T_{e}}{z_p}+ \sum_{p \in \bar T_{e} \setminus \tilde T_{e}}{z_{w_{e,p}}} - z_{e} \leq |T_{e}|-1, \quad
 \forall  w_{e,p} \in \P_{e,p}, \; \forall T_{e}: \tilde e \in \tilde T_e, \; \forall e \in E: \tilde e \in \Ri_e,\label{k5}
\end{align}
where $\P_{e,p}$ is defined by~\eqref{e8}.
Next, using Fourier-Motzkin elimination, we project out $z_{\tilde e}$, $\tilde e \in \tilde E$ from System~\eqref{k1}-\eqref{k5}.
To this end, it suffices to consider the following cases:

\medskip

\emph{Inequalities~\eqref{k1}-\eqref{k4}}: first consider inequalities~\eqref{k1} and~\eqref{k2}; projecting out $z_{\tilde e}$ we obtain $z_{p_{\tilde e}}\geq 0$ 
and $z_{q_{\tilde e}} \geq 0$,
both of which are implied by inequalities~\eqref{proj0} and~\eqref{proj1}. Next, consider inequalities~\eqref{k2} and~\eqref{k3}; projecting out $z_{\tilde e}$ we obtain $z_{p_{\tilde e}}\leq 1$ and $z_{q_{\tilde e}} \leq 1$; again these inequalities are implied by inequalities~\eqref{proj0} and~\eqref{proj1}. Finally, we consider inequalities~\eqref{k2}
and~\eqref{k4}. Projecting out $z_{\tilde e}$ we obtain
$$
z_e \leq z_{p_{\tilde e}}, \; z_e \leq z_{q_{\tilde e}}, \quad \forall e \in E: \; \tilde e \in \Ri_e.
$$
Recall that $\tilde e = p_{\tilde e} \cup q_{\tilde e}$ for some $p_{\tilde e}, q_{\tilde e} \in V \cup E$. Since $\tilde e \in \Ri_e$, from the definition of $\Ri_e$ it follows that, if $|p_{\tilde e}| > 1$ (resp. $|q_{\tilde e}| >1$), then $p_{\tilde e} \in \Ri_e$ (resp. $q_{\tilde e} \in \Ri_e$). Hence, the above inequalities are present among inequalities~\eqref{proj1}.

\emph{Inequalities~\eqref{k3}-\eqref{k5}}: We now project out $z_{\bar e}$, $\bar e \in \tilde T_e$ from inequalities~\eqref{k5}.
First note that if we use inequality~\eqref{k1} to project some $z_{\bar e}$ from~\eqref{k5}, we obtain a redundant inequality, since the left-hand side of the
resulting inequality cannot exceed the right-hand side. Next, consider inequalities~\eqref{k4} for some $\bar e \in \tilde T_e$; If $|\{e \in E: \bar e \in \Ri_e\}| =1$,
then projecting $z_{\bar e}$ from~\eqref{k4} and~\eqref{k5} gives the redundant inequality:
$$\sum_{p \in \tilde T_e \setminus \{e\}}{z_p}+\sum_{p \in T_e \setminus \bar T_{e}}{z_p}+ \sum_{p \in \bar T_{e} \setminus \tilde T_{e}}{z_{w_{e,p}}} \leq |T_{e}|-1.$$
Hence, we use inequalities~\eqref{k4} to project out an artificial variable $z_{\bar e}$ only if $|\{e \in E: \bar e \in \Ri_e\}| \geq 2$.

Now let $\tilde T_e = \tilde T^1_e \cup \tilde T^2_e$, for any (possibly empty) $\tilde T^1_e$ and $\tilde T^2_e$.
Let $\bar e = p_{\bar e} \cup q_{\bar e}$ for any $\bar e \in \tilde T^1_e$, where $p_{\bar e}, q_{\bar e} \in V \cup E$.
We now project out artificial variables $z_{\bar e}$, $\bar e \in \tilde T_e$ from inequalities~\eqref{k5}, using inequalities~\eqref{k3} for all $\bar e \in \tilde T^1_e$,
and using inequalities~\eqref{k4} for all $\bar e \in \tilde T^2_e$, to obtain:
\begin{equation}\label{semifl}
\sum_{\bar e \in \tilde T^1_e}{(z_{p_{\bar e}} + z_{q_{\bar e}})} +\sum_{\bar e \in \tilde T^2_e}{z_{f_{\bar e}}}+\sum_{p \in T_e \setminus \bar T_{e}}{z_p}+ \sum_{p \in \bar T_{e} \setminus \tilde T_{e}}{z_{w_{e,p}}} - z_{e} \leq |T_{e}|+|\tilde T^1_e|-1,
\end{equation}
for all  $T_{e}$ such that  $\tilde T_e \neq \emptyset$ and for all $e \in E$ such that $\Ri_e \cap \tilde E \neq \emptyset$, where we define $f_{\bar e}$ an edge in $E$
different from $e$ such that $\bar e \in \Ri_{f_{\bar e}}$. Now consider the flower partition $T_e$ of $e$ and define a new flower partition $T'_e$ of $e$, where each $\bar e \in \tilde T^1_e$ is replaced by $p_{\bar e}, q_{\bar e}$. Since by definition $\bar e = p_{\bar e} \cup q_{\bar e}$, it follows that $|T'_e| = |T_{e}|+|\tilde T^1_e|$. Hence, any inequality of the form~\eqref{semifl} is present among inequalities~\eqref{proj2}, or by Remark~\ref{rem1} is implied by these inequalities and this completes the proof.
\end{proof}

We conclude this paper by showing via a simple example that the extended flower relaxation (even the flower relaxation)
often provides a stronger relaxation than any RMC.

\begin{example}
Consider the hypergraph $G= (V, E)$ with $V = \{v_1, \ldots, v_8\}$ and $E = \{e_0, e_1, e_2, e_3, e_4\}$, where $e_0 = \{v_1,\ldots,v_4\}$,
$e_1 = \{v_1,v_2,v_5\}$, $e_2 = \{v_2,v_3,v_6\}$, $e_3 = \{v_3,v_4,v_7\}$, and $e_4 = \{v_1,v_4,v_8\}$. Then the flower relaxation of $\S_G$ contains
the following six flower inequalities centered at $e_0$:
\begin{align}
&\sum_{v \in e_0 \setminus e_i}{z_v}+z_{e_i}-z_{e_0} \leq 2, \quad i \in \{1,\ldots, 4\}\label{fl1}\\
&z_{e_i}+z_{e_j}-z_{e_0} \leq 1, \quad (i,j) \in \{(1,3), (2,4)\}.\label{fl2}
\end{align}
Now consider an RMC of $\S_G$. Then the above inequalities are implied by this RMC, only if $\Ri_{e_0} \supseteq \{\{v_1, v_2\}, \{v_2, v_3\}, \{v_3, v_4\}, \{v_1, v_4\}\}$.
However this is not possible as the set of all possible recursive sequences of $z_{e_0}$ are given by
\begin{itemize}
\item [(i)] $\{\{v_1, v_2, v_3, v_4\}, \{v_i, v_j, v_k\}, \{v_i, v_j\}\}$
for all $i\neq j \neq k \in \{1,\ldots,4\}$, which implies only one of the flower inequalities of~\eqref{fl1}, centered at $e_0$, and
\item [(ii)] $\{\{v_1, v_2, v_3, v_4\}, \{v_i, v_j\}, \{v_k, v_l\}\}$ for all $i\neq j \neq k \neq l \in \{1,\ldots,4\}$, which implies two flower inequalities of~\eqref{fl1} and one flower inequality of~\eqref{fl2}.
\end{itemize}
\end{example}

\bibliographystyle{plain}
\bibliography{aida,biblio}

\end{document}